\documentclass[12pt]{amsart}
\usepackage{amsmath,amsthm,amssymb,latexsym,mathrsfs}
\newcommand{\cal}[1]{\mathscr{#1}}

\newcommand{\fk}[1]{\mathfrak{#1}}
\newcommand{\GK}{\mathrm{GKdim}}
\newcommand{\bC}{\mathbb{C}}

\newcommand{\bZ}{\mathbb{Z}}
\newcommand{\bD}{\mathbb{D}}

\newcommand{\gr}{\mathrm{gr}}

\newcommand{\Spec}{\mathrm{Spec}}

\newcommand{\Ext}{\mathrm{Ext}}
\newcommand{\Hom}{\mathrm{Hom}}
\newcommand{\RHom}{\mathrm{RHom}}

\newtheorem{thm}{Theorem}[section]
\newtheorem*{thm*}{Theorem}
\newtheorem{cor}[thm]{Corollary}
\newtheorem{lem}[thm]{Lemma}
\newtheorem{prop}[thm]{Proposition}
\theoremstyle{definition}
\newtheorem{defn}[thm]{Definition}
\newtheorem{rmk}[thm]{Remark}
\newtheorem{eg}[thm]{Example}

\numberwithin{equation}{section}

\usepackage[backend=bibtex,style=ieee,sorting=nyt]{biblatex}
\bibliography{holonomic}

\title[Holonomic Cherednik modules, I]{Holonomic modules over Cherednik algebras, I}
\author{Daniel Thompson}
\email{dthomp@math.mit.edu}
\date{\today}
\begin{document}
\begin{abstract}
The goal of this paper is to generalize several basic results from the theory of $\cal{D}$-modules to the representation theory of rational Cherednik algebras.  We relate characterizations of holonomic modules in terms of singular support and Gelfand-Kirillov dimension.  We study pullback, pushforward, and dual on the derived category of (holonomic) Cherednik modules for certain classes of maps between varieties.  We prove, in the case of generic parameters for the rational Cherednik algebra, that pushforward with respect to an open affine inclusion preserves holonomicity.
\end{abstract}
\maketitle
\section{Introduction}
The rational Cherednik algebra $H_c$ of a finite subgroup $W$ of $GL_n(\bC)$ is a universal deformation of the skew-product algebra $\cal{D}(\bC^n) \rtimes \bC W$ of the Weyl algebra (algebra of polynomial differential operators on $\bC^n$) with the group algebra of $W$ (see \cite[Theorem 2.13]{E04}).  As such, it is a family of algebras over a space of parameters $c$, and in particular when we specialize the parameters to the numerical value $c=0$ we just obtain $H_0 = \cal{D}(\bC^n) \rtimes \bC W$.  Thus, we may view the representation theory of $H_c$ as a deformation of the theory of $W$-equivariant $\cal{D}$-modules on affine space.

The foundational paper \cite{GGOR} studies a special category of representations of $H_c$ which deforms the category of \emph{$\cal{O}$-coherent, graded} $W$-equivariant $\cal{D}$-modules.  For parameters $c$ outside a countable collection of hyperlanes, this category is identical to the category of representations of $W$.  In particular, the same is true of this category if $c=0$.

Understanding the general representation theory of $H_c$ remains a very difficult problem.  I. Losev \cite{L} has introduced the notion of \emph{holonomic} representations of certain algebras, including rational Cherednik algebras, and it is the category of such representations that we study in the present paper.  For generic parameters, we expect the theory of holonomic representations of $H_c$ to be similar to the theory of $W$-equivariant holonomic $\cal{D}$-modules on $\bC^n$ (the $c=0$ case).  This expectation comes despite the fact that the Morita equivalence class of $H_c$ may depend of $c$ even for generic $c$: see \cite{BEG04} which classifies $H_c$ up to Morita equivalence for $W$ equal to the symmetric group.  On the other hand, for special parameters, the presence of $\cal{O}$-coherent modules with less than full support as in \cite{GGOR} already shows that the category of holonomic Cherednik modules is more complicated than the theory of $W$-equivariant holonomic $\cal{D}$-modules.

More generally, one can study $\cal{D}$-modules on smooth varieties.  P. Etingof \cite{E04} has defined a sheaf of Cherednik algebras on a smooth variety with the action of a finite group.  The parameter space is more complicated (in particular, there is a global ``twisting'' parameter), but again by \cite[Theorem 2.13]{E04} this sheaf of algebras is a universal deformation of the skew-product of the sheaf of differential operators on the variety with the group algebra.

The goal of this paper is to generalize several basic results from the theory of $\cal{D}$-modules to the current setting.  The structure of the paper is as follows.  After introducing the basic objects of study in section 2, we investigate in section 3 characterizations of holonomic modules in terms of singular support and Gelfand-Kirillov dimension.  In section 4 we discuss the global setting of Cherednik modules on a variety, and we introduce pullback and pushforward for certain equivariant maps between varieties.  We also introduce the Verdier dual functor on the derived category, and discuss an analog of Kashiwara's theorem.  In section 5 we prove, in the case of generic parameters for the rational Cherednik algebra, that pushforward with respect to open affine inclusions takes holonomic modules to holonomic modules over the rational Cherednik algebra.

\subsection*{Acknowledgements}
The author would like to thank Pavel Etingof and Gwyn Bellamy for reading earlier drafts of this paper and for discussions along the way.  This material is based upon work supported by the National Science Foundation Graduate Research Fellowship Program under Grant No. 1122374.

\section{Setup}
\subsection{Rational Cherednik algebra}
Let $\fk{h}$ be a complex vector space of dimension $r$, and let $W$ be a finite group acting linearly on $\fk{h}$ (we do not require the action to be faithful).  We say an element of $s\in W$, $s\neq 1$, is a \emph{reflection} if it fixes pointwise a codimension 1 hyperplane.  Let $\cal{S}\subset W$ be the set of reflections, and choose a function $c:\cal{S}\to \bC$ which is constant on $W$-conjugacy classes in $\cal{S}$.  Finally, for $s\in \cal{S}$ choose eigenvectors $\alpha_s\in \fk{h}^*$ and $\alpha_s^\vee \in \fk{h}$ for the action of $s$, both with eigenvalue different from 1, and normalized so that $\langle \alpha_s,\alpha_s^\vee\rangle=2$ where $\langle \cdot, \cdot \rangle$ is the natural pairing between $\fk{h}^*$ and $\fk{h}$.  We recall, \cite{EM}, the definiton of the rational Cherednik algebra $H_c(W,\fk{h}) := H_{1,c}(W,\fk{h})$.  It is defined to be the quotient of the smash-product algebra $T(\fk{h}\oplus \fk{h}^*) \rtimes \bC W$ (here $T(V)$ denotes the tensor algebra of $V$) by relations of the following form:
$$[x,x']=0, \quad [y,y']=0, \quad [y,x] = \langle x, y \rangle - \sum_{s\in \cal{S}} c(s) \langle \alpha_s,y\rangle \langle x,\alpha_s^\vee\rangle s,$$
where $x,x'\in \fk{h}^*, y,y'\in \fk{h}$.
When there is no confusion about $W,\fk{h}$ we will just write $H_c$.

We will assume the parameter $c$ is arbitrary unless otherwise specified.

\subsection{Filtrations and PBW theorem}
Let us now introduce two filtrations on $H_c$.  Since the algebra is generated by $\fk{h}^*, \fk{h},$ and $\bC W$, it will suffice to specify the degrees of the generators to define these filtrations.  The \emph{geometric filtration} is given by putting $\deg(\fk{h}^*)=0$ and $\deg(\fk{h})=1$.  The \emph{Bernstein filtration} is given by putting $\deg(\fk{h}^*)=\deg(\fk{h})=1$.  In both filtrations we set $\deg(\bC W) = 0$.

It is clear that, with respect to either of these filtrations, we have a natural surjective homomorphism
$$\bC[\fk{h}\oplus \fk{h}^*] \rtimes \bC W \to \gr H_c.$$
In fact, these maps are both isomorphisms.  This claim is equivalent to the standard fact that the algebra $H_c$ has a basis given by the set of elements of the form
$$g \prod_{i=1}^r y_i^{m_i} \prod_{i=1}^r x_i^{n_i}$$ 
where $y_i$ is a basis of $\fk{h}$, $x_i$ is a basis of $\fk{h}^*$, and $g$ is an element of $W$.

\subsection{GK dimension}
We recall the definition of Gelfand-Kirillov dimension of a module.  For proofs of standard facts about good filtrations we refer the reader to \cite[Appendix D.1]{HTT}.  Suppose $M$ is a finitely generated $H_c$-module.  Then with respect to either filtration on $H_c$, there exists a \emph{good filtration} on $M$, that is, an ascending filtration $M=\bigcup_{j=0}^\infty M_j$ such that $\gr M$ is finitely generated over $\gr H_c$.

Now consider the Bernstein filtration on $H_c$.  There is a polynomial $h_M(j)$ which agrees with $\dim M_j$ for all sufficiently large $j$.  The \emph{Gelfand-Kirillov dimension} is defined to be the degree of $h_M(j)$.  It does not depend on the choice of filtration (though the polynomial $h_M(j)$ in general does).

We write $\GK(M)$ for the GK dimension of $M$.

\subsection{Singular support and holonomic modules}
Losev \cite[Section 1.4]{L} gives a definition of holonomicity of a finitely generated $H_c$-module $M$, which ostensibly depends on the filtration chosen on $H_c$.  Our first results will show that his definition yields the same notion of holonomicity if we take either the Bernstein or geometric filtration.

Let $Z(A)$ denote the center of an algebra $A$.  Note first that $Z(\gr H_c) = \bC[\fk{h}\oplus \fk{h}^*]^W$, and so $\Spec\, Z(\gr H_c) = (\fk{h}\oplus \fk{h}^*)/W$ is a Poisson variety with finitely many symplectic leaves, each of which has finite fundamental group.  More specifically, the leaves of $(\fk{h}\oplus \fk{h}^*)/W$ are in one to one correspondence with the conjugacy classes of parabolic subgroups of $W$.  Recall that a subgroup $W'\subset W$ is parabolic (with respect to the action on $\fk{h}$) if it is the stabilizer of a point of $\fk{h}$.  For a parabolic subgroup $W'\subset W$, the corresponding leaf $\cal{L}_{W'}$ is $(\fk{h}\oplus \fk{h}^*)^{W', W'-\mathrm{reg}} / N(W')$, where $N(W')$ is the normalizer of $W'$ in $W$, and where $$(\fk{h}\oplus \fk{h}^*)^{W', W'-\mathrm{reg}} = \{v\in \fk{h}\oplus \fk{h}^* | \mathrm{Stab}_{W}(v) = W'\}.$$  Now $(\fk{h}\oplus \fk{h}^*)^{W', W'-\mathrm{reg}}$ is simply connected, as it is the complement in the complex vector space $(\fk{h}\oplus \fk{h}^*)^{W'}$ of subspaces of codimension at least 2, so we have $\pi_1(\cal{L}_{W'}) = N(W')/W'$, a finite group.

Now let us recall the definition of holonomicity.  Equip $M$ with a good filtration, and define the singular support to be the set-theoretic support in $(\fk{h}\oplus \fk{h}^*)/W$ of $\gr M$, thought of as a module over $Z(\gr H_c)$.  This variety does not depend on the choice of good filtration.  If we take the Bernstein (resp., geometric) filtration on $H_c$, we obtain the \emph{arithmetic} (resp., \emph{geometric}) singular support $V_a(M)$ (resp., $V(M)$).  We say $M$ is \emph{holonomic} with respect to the Bernstein (resp., geometric) filtration if the smooth locus of $V_a(M)\cap \cal{L}$ (resp., $V(M)\cap \cal{L}$)  is isotropic in every symplectic leaf $\cal{L}$ of $(\fk{h}\oplus \fk{h}^*)/W$.  We will see in the next section that being holonomic with respect to either filtration on $H_c$ is the same condition, so we say $M$ is \emph{holonomic} if it is holonomic with respect to either.

The following proposition shows, in particular, that $\dim V_a(M) = \dim V(M)$.  Note that $\Ext^i_{H_c}(M,H_c)$ is a finitely generated right $H_c$-module because $H_c$ is left and right Noetherian.

\begin{prop} \label{prop:exts} Let $M$ be a nonzero, finitely generated $H_c$-module.  Let $d(M)$ be the dimension of the arithmetic or geometric singular support of $M$ in $(\fk{h}\oplus \fk{h}^*)/W$, and let $j(M)=\min\{i\ |\ \Ext^i_{H_c}(M,H_c)\neq 0\}$.  Then we have
  \begin{enumerate}
  \item $d(M) = 2r - j(M)$,
  \item $d(\Ext^i_{H_c}(M,H_c)) \leq 2r-i$ for all $i\in \bZ$,
  \item $d(\Ext^{j(M)}_{H_c}(M,H_c))=d(M)$.
  \end{enumerate}
\end{prop}
\begin{proof} The usual proof applies (see \cite[Theorem D.4.3]{HTT}) because $$\Ext^i_{\bC[\fk{h}\oplus \fk{h}^*]\rtimes \bC W}(N,\bC[\fk{h}\oplus \fk{h}^*]\rtimes \bC W) \simeq \Ext^i_{\bC[\fk{h}\oplus \fk{h}^*]}(N,\bC[\fk{h}\oplus \fk{h}^*]).$$
\end{proof}

\subsection{Losev's results}
Because $(\fk{h}\oplus \fk{h}^*)/W$ has finitely many symplectic leaves, each of which has finite fundamental group, Losev's Theorem 1.3 of \cite{L} ensures that $H_c$ is finite length as a bimodule over itself.  This allows us to state his main results for Cherednik modules as follows.

First, there is an analogue of Bernstein's inequality:
\begin{thm}[Theorem 1.1 of \cite{L}] \label{thm:bernst}
  Let $M$ be a nonzero, finitely generated $H_c$-module with annihilator $I$.
  \begin{enumerate}
  \item We have $2 \dim V(M) \geq \dim V(H_c/I)$.
  \end{enumerate}
  If we assume further that $M$ is simple, then
  \begin{itemize}
  \item[(2)] if $\cal{L}$ is a symplectic leaf of $(\fk{h}\oplus \fk{h}^*)/W$ which contains a nonempty open subset of $V(M)$, then $V(H_c/I) = \overline{\cal{L}}$ and $V(M) \cap \cal{L}$ is coisotropic in $\cal{L}$, and
  \item[(3)] every irreducible component of $V(M)$ intersects nontrivially with $\cal{L}$.
  \end{itemize}
\end{thm}
A simple corollary of the latter part is
\begin{cor} \label{cor:los} If $M$ is simple, then the only leaf $\cal{L}$ which contains a nonempty open subset of $V(M)$ as in (2) is the unique maximal leaf of $V(H_c/I)$. Also, $V(M)$ is the closure of its intersection with this $\cal{L}$. \end{cor}

More is true if $M$ is holonomic.
\begin{thm}[Theorem 1.2 of \cite{L}] \label{thm:loshol}
  Let $M$ be a holonomic $H_c$-module with annihilator $I$.
  \begin{enumerate}
  \item $M$ has finite length, and we have $2 \dim V(M) = \dim V(H_c/I)$.
  \item Furthermore, if $M$ is simple then $V(M)$ is equidimensional.
  \end{enumerate}
\end{thm}

\section{GK dimension and holonomicity}
\subsection{Regular values}
As noted in the introduction, the respresentation theory of $H_c$ can change dramatically based on the choice of parameters.  As such, it is useful to draw a distinction between properties which only hold generically in the parameter space, and those which are true for arbitrary choices.  First let us specify precisely what we mean by ``generic'' parameters.

Let us recall the following notions from \cite{BEG03}.

\begin{defn}
  A parameter $c$ is called \emph{regular} if the category $\cal{O}_c$ studied in \cite{GGOR} is semisimple.
\end{defn}

\begin{thm}
  The following are equivalent:
  \begin{enumerate}
  \item $c$ is regular,
  \item the associated Hecke algebra is semisimple,
  \item the Cherednik algebra $H_c$ is simple.
  \end{enumerate}
\end{thm}
\begin{proof} A proof of this result appears in \cite[Theorem 6.6]{BC}.  The hypothesis of that theorem, that the dimension of the Hecke algebra agrees with $|W|$, is now known to hold for any finite group $W$: see \cite{E16}.
\end{proof}

This condition is Weil-generic in the parameter space, that is, the complement of the set of regular values is contained in a countable collection of hyperplanes.  We remind the reader that $c=0$ is a regular value.

\begin{eg} Let us describe the set of regular values if $W$ is a real reflection group and for constant parameter $c$.  Let $d_i$, for $i=1, \ldots, \dim \fk{h}$ be the degrees of $W$.  (The Chevalley--Shepard--Todd theorem says that $\bC[\fk{h}]^W$ is a polynomial algebra.  It has an algebraically independent set of generators which are homogeneous of degrees $d_i$.)  Now $c$ is regular if and only if for each $i$, $c\neq m/d_i$ for any integer $m$ not divisible by $d_i$.
\end{eg}
  
\subsection{Generic parameters}
Assume for the remainder of this section that the parameter $c$ is regular, so the annihilator of any nonzero $H_c$-module is trivial.  

\begin{prop} \label{prop:gkdim} Let $M$ be a nonzero, finitely generated $H_c$-module.  The following are equivalent
  \begin{enumerate}
  \item $M$ has GK dimension $r$,
  \item $M$ is holonomic with respect to the Bernstein filtration,
  \item $M$ is holonomic with respect to the geometric filtration.
  \end{enumerate}
\end{prop}
\begin{proof}  First let us note that $\GK(M)=\dim V_a(M)$.  This follows from the Hilbert-Serre theorem \cite[Ch. 1, Theorem 7.5]{Ha} if we view $\gr M$ as a graded $\bC[\fk{h}\oplus \fk{h}^*]$-module, for $V_a(M)$ is the image of the support of $\gr M$ under the finite map $\fk{h}\oplus\fk{h}^* \to (\fk{h}\oplus\fk{h}^*)/W$.

  Suppose $M$ is holonomic with respect either filtration.  Then part (1) of Theorem \ref{thm:loshol}, that is to say Losev's work, already implies that $\dim V_a(M) = \dim V(M)$ (Proposition \ref{prop:exts}) has dimension $r$.  Thus $M$ has GK dimension $r$.

  Conversely, suppose $M$ has GK dimension $r$.  Just as for holonomic $\cal{D}$-modules, we may define the arithmetic (resp., geometric) singular $r$-cycle of $M$.  Note that Theorem \ref{thm:bernst} implies that the singular support of any subquotient of $M$ has dimension $r$.  The singular cycle is additive over short exact sequences, which shows that $M$ has finite length, and so we may assume that $M$ is simple.  Then again by Corollary \ref{cor:los}, $V_a(M) = \overline{Y}$ (resp., $V(M)= \overline{Y}$) where $Y$ is an isotropic subvariety of the maximal symplectic leaf $\cal{L}$ of $(\fk{h}\oplus \fk{h}^*)/W$.  Let $V'(M)$ be the preimage of $V_a(M)$ (resp., $V(M)$) in $\fk{h}\oplus \fk{h}^*$.  Then $V'(M)$ is isotropic, being the closure of the preimage of $Y$, so by \cite[Proposition 1.3.30]{CG} the intersection of $V'(M)$ with the preimage of each symplectic leaf is again isotropic.  Hence $M$ is holonomic with respect to either filtration.
\end{proof}

As a consequence of this proposition, in the case $c=0$, we see that our notion of holonomic modules coincides with the usual definition for $W$-equivariant $\cal{D}$-modules.

\begin{cor} Let $M$ be a nonzero, finitely generated $H_c$-module.  Then $M$ is holonomic if and only if $\Ext^i_{H_c}(M,H_c)= 0 $ for $i\neq r$.
\end{cor}
\begin{proof}
  If $M$ is holonomic, then part (1) of Proposition \ref{prop:exts} implies that $\Ext^i_{H_c}(M,H_c)= 0 $ for $i< r$.  Finally, if $i> r$ then by part (2) we have $d(\Ext^i_{H_c}(M,H_c))<r$ and so by Bernstein's inequality $\Ext^i_{H_c}(M,H_c)=0$.

  If, on the other hand, $\Ext^i_{H_c}(M,H_c)= 0 $ for $i\neq r$ then part (1) of Proposition \ref{prop:exts} implies that $d(M)=r$.  This shows $M$ is holonomic by Proposition \ref{prop:gkdim}.
\end{proof}

\begin{rmk}
Below we will define the Verdier dual functor $\RHom_{H_c}(\cdot,H_c)$ on the category of holonomic modules, which for generic parameters will just be $\Ext^r_{H_c}(M,H_c)$.  Thus we have $M\simeq \Ext^r_{H_c^\mathrm{op}}(\Ext^r_{H_c}(M,H_c),H_c)$.
\end{rmk}

\subsection{Arbitrary parameters}
\begin{prop} \label{prop:gkdimsing} Let $M$ be a nonzero, finitely generated $H_c$-module with annihilator $I$.
  \begin{itemize} 
  \item[(i)] \emph{(Theorems 1.2, 1.3 of \cite{L}).}
    If $M$ is holonomic (with respect to either the Bernstein or the geometric filtration) then $\GK(M)=\frac{1}{2}\dim V(H_c/I)$.
  \item[(ii)]
    If $M$ is simple and $\GK(M)=\frac{1}{2}\dim V(H_c/I)$ then $M$ is holonomic (again with respect to either filtration).
  \end{itemize}
\end{prop}
\begin{proof}
  First assume $M$ is holonomic.  Then part (1) of Theorem \ref{thm:loshol} already implies that $\dim V_a(M) = \dim V(M)$ (Proposition \ref{prop:exts}) is equal to $\frac{1}{2}\dim V(H_c/I)$.  Thus $\GK(M)=\frac{1}{2}\dim V(H_c/I)$.

  Now suppose $M$ is simple and has GK dimension $\frac{1}{2}\dim V(H_c/I)$.  Then again by Corollary \ref{cor:los}, $V_a(M) = \overline{Y}$ (resp., $V(M)= \overline{Y}$) where $Y$ is an isotropic subvariety of the maximal symplectic leaf $\cal{L}$ of $V_a(H_c/I)$ (resp., $V(H_c/I)$).  Let $V'(M)$ be the preimage of $V_a(M)$ (resp., $V(M)$) in $\fk{h}\oplus \fk{h}^*$.  Then $V'(M)$ is isotropic, being the closure of the preimage of $Y$, so by \cite[Proposition 1.3.30]{CG} the intersection of $V'(M)$ with the preimage of each symplectic leaf is again isotropic.  Hence $M$ is holonomic.
\end{proof}

\begin{eg}
Here is an example which shows that the condition that $M$ be simple is necessary for $\GK(M)=\frac{1}{2}\dim V(H_c/I)$ to imply that $M$ is holonomic.  Consider the case $\fk{h}=\bC\oplus \bC$, $W=\bZ/2\bZ$ acting by reflection on the first copy of $\bC$ and trivially on the second one.  Then $H_c(W,\fk{h}) = H_c(W,\bC) \otimes \cal{D}(\bC)$.  Take $c=1/2$.  Now consider the following module for $H_c(W,\fk{h})$:
$$M(\bC_-) \boxtimes \cal{O} \oplus \bC \boxtimes \cal{D}(\bC),$$
where $M(\bC_-)$ is the irreducible Verma module for $H_c(W,\bC)$, $\bC$ is the irreducible finite dimensional module respectively for $H_c(W,\bC)$, and $\cal{O}$ is the $\cal{D}(\bC)$-module given by polynomial functions on $\bC$.  This module has Gelfand-Kirillov dimension 2 and trivial annihilator, but is not holonomic.
\end{eg}

\subsection{Quantized symplectic resolutions}
The proof of Proposition \ref{prop:gkdimsing} also applies to the setting of quantized symplectic resolutions.  Before we express the result in this case, let us recall the set-up from \cite{L}.

Let $\cal{A} = \bigcup_{i= 0}^\infty \cal{A}_i$ be an associative $\bC$-algebra with an ascending filtration such that the associated graded algebra $A=\gr \cal{A}$ is commutative and finitely generated.  Let $d$ be a positive integer such that $[\cal{A}_i,\cal{A}_j] \subset \cal{A}_{i+j-d}$, so we have a natural degree $-d$ Poisson bracket on $A$.  Assume that $X$ admits a conical symplectic resolution of singularities $\rho:\tilde{X}\to X$ and that $\cal{A} = \Gamma(\cal{D})$ for some filtered quantization $\cal{D}$ of $\tilde{X}$ (i.e., a filtered quantization of the sheaf $\cal{O}_{\tilde X}$ in the conical topology).

The definition of holonomic modules is analogous to the case of the Cherednik algebra.  If $M$ is a finitely generated $\cal{A}$-module, we may equip $M$ with a good filtration, and define the singular support $V(M)$ to be the set-theoretic support in $X$ of $\gr M$.  Again this subvariety is independent of the choice of good filtration.  We say $M$ is \emph{holonomic} if the smooth locus of $V(M)\cap \cal{L}$ is isotropic in every symplectic leaf $\cal{L}$ of $(\fk{h}\oplus \fk{h}^*)/W$.

\begin{prop} \label{prop:sympres} Let $M$ be a nonzero, finitely generated $\cal{A}$-module with annihilator $\cal{I}$.
  \begin{itemize}
  \item[(i)] \emph{(Theorems 1.2, 1.3 of \cite{L}).} If $M$ is holonomic then $\dim V(M)=\frac{1}{2}\dim V(\cal{A}/\cal{I})$.
  \item[(ii)] If $M$ is simple with $\dim V(M)=\frac{1}{2}\dim V(\cal{A}/\cal{I})$ then $M$ is holonomic.
  \end{itemize}
\end{prop}
\begin{proof}
  The proof of the first implication is the same as the proof of Proposition \ref{prop:gkdimsing}, so suppose $M$ is simple and has $\dim V(M) = \frac{1}{2}\dim V(\cal{A}/\cal{I})$.  Then by \cite[Theorem 1.1]{L}, $V(M)$ is the closure of an isotropic subvariety $Y$ of the maximal symplectic leaf $\cal{L}$ of $V(\cal{A}/\cal{I})$.  Now $\rho^{-1}(Y)$ is isotropic in $\tilde{X}$ by \cite[Lemma 5.1]{L} so the closure $Z$ of $\rho^{-1}(Y)$ is isotropic as well.  Finally, $\rho(Z) = V(M)$ since $\rho$ is proper, so \cite[Lemma 5.1]{L} shows that $M$ is holonomic.
\end{proof}

\section{Functoriality}

\subsection{Holonomicity for Cherednik modules on a variety}

Suppose $X$ is a smooth variety with the action of a finite group $W$ for which the quotient variety $X/W$ exists.  Define the set $S(X)$ of \emph{reflections} of $X$ to be the set of pairs $(w,Z)$ such that $w\in W$ and $Z$ is an irreducible component of $X^w$ having codimension 1 in $X$.  Let $c:S(X)\to \bC$ be a $W$-invariant function.  Finally, we must choose an element $\omega \in \mathbb{H}^2(X, \Omega_X^{\geq 1})^W$, where $\Omega_X^{\geq 1}$ is the two-term subcomplex $\Omega_X^1 \to (\Omega_X^2)^\mathrm{cl}$, concentrated in degrees 1 and 2, of the algebraic De Rham complex and where $(\Omega_X^2)^\mathrm{cl}$ denotes the subsheaf of closed forms in $\Omega_X^2$.  Sheaves of twisted differential operators are classified by the space $\mathbb{H}^2(X, \Omega_X^{\geq 1})$, see section 2 of \cite{BB}.  Let us write $\cal{D}^\omega_X$ to refer to the sheaf of twisted differential operators corresponding to $\omega$.

We have the sheaf of Cherednik algebras $H_{c,\omega}(W,X):=H_{1,c,\omega}(W,X)$ on $X$ in the $W$-equivariant topology (alternatively, as a sheaf on $X/W$), defined in \cite{E04}.  If we write $D = \bigcup_{(w,Z)\in S(X)} Z$ and $j:X\setminus D \to X$ for the inclusion, then $H_{c,\omega}(W,X)$ is defined as a subalgebra of the sheaf $j_* j^*(\cal{D}^\omega_X \rtimes \bC W)$ generated locally by $\cal{O}_X$, $\bC W$, and Dunkl operators $D_y$ associated to vector fields $y$.  Again when there is no confusion about $W,X$ we will just write $H_{c,\omega}$, or if $\omega =0$ just $H_c$.  This sheaf has a natural filtration, the analogue of the geometric filtration above, for which $\gr H_{c,\omega} = p_*(\cal{O}_{T^*X}\rtimes \bC W)$, where $p:T^*X \to X$ is the projection.  Thus we may define the singular support $V(M)\subset T^*X/W$ of any sheaf of modules $M$ which is coherent over $H_{c,\omega}$.  For such $W,X$ we have a notion of holonomic modules.

To define these, we work locally.  Given a $W$-stable affine open inclusion $j:U\to X$, we may choose a good filtration on the restriction $M(U)$ of $M$ to $U$.  Note, the restriction $H_{c,\omega}(W,X)|_U$ is just $H_{j^* c,j^* \omega}(W,U)$, where $j^* c$ is the restriction of $c$ to the set of reflections which intersect with $U$ and $j^* \omega$ is the restriction of $\omega$ to $U$.  Then we glue together $V(M)\subset T^*X/W$ from the (reduced) subvarieties of $T^*U / W$ corresponding to $$\mathrm{Ann}_{Z(\gr H_{j^* c,j^* \omega}(W,U))} (\gr M(U)),$$ where $U$ ranges over an affine open cover of $X$.  As before, we say $M$ is \emph{holonomic} if the smooth locus of $V(M)\cap \cal{L}$ is isotropic in every symplectic leaf $\cal{L}$ of $T^*X/W$.

We plan to show in the next paper that the main results of \cite{L} hold in this larger generality.


Finally, let us mention the \emph{stabilizer stratification} of $X$.  The strata here are in correspondence with conjugacy classes of parabolic subgroups of $W$ as follows.  Let $Par(W,X)$ denote the set of subgroups $W'\subset W$ that occur as isotropy groups of points of $X$.  We refer to such subgroups as \emph{parabolic} subgroups of $W$ with respect to its action on $X$.  Now for a parabolic $W'\subset W$, the stratum corresponding to its conjugacy class is $X_{W'}=W\cdot X^{W',W'-\mathrm{reg}}$, where
$$X^{W',W'-\mathrm{reg}}= \{x\in X|\mathrm{Stab}_W(x)=W'\}.$$
Note furthermore that the stabilizer stratification on $X$ induces a stratification on $X/W$ to which we give the same name.

\subsection{Left and right modules} \label{sec:lr}
There is a natural equivalence between categories of left and right Cherednik modules for $(W,X)$, though the parameters for these two categories will in general be different.  Throughout this paper, sheaves of $H_{c,\omega}$-modules on a variety are always assumed to be quasicoherent over its structure sheaf.

Given a line bundle $\cal{L}$ on $X$, let $\omega_\cal{L}\in \mathbb{H}^2(X, \Omega_X^{\geq 1})$ be given by its curvature. Let $K_X$ be the canonical line bundle on $X$, which is naturally $W$-equivariant, and say $\omega_\mathrm{can} = \omega_{K_X}$.  It is well-known that $K_X$ carries a right action of $\cal{D}_X$, which leads to an isomorphism $(\cal{D}_X)^\mathrm{op}\simeq \cal{D}_X^{\omega_\mathrm{can}}$, and more generally $$(\cal{D}_X^\nu)^\mathrm{op}\simeq \cal{D}_X^{\omega_\mathrm{can} - \nu}.$$  If $y\in \Gamma(U,TX)$, then under this isomorphism a twisted Lie derivative $\mathbb{L}_y \in (\cal{D}_X^\nu)^\mathrm{op}(U)$ goes to the negative twisted Lie derivative for $y$: $-\mathbb{L}_y \in \cal{D}_X^{\omega_\mathrm{can} - \nu}(U)$.  See \cite[Section 2.2]{E04}.

Set $\overline{c}(w,Z)=c(w^{-1},Z)$.  For each $(w,Z)\in S(X)$, let $\omega_Z = \omega_{\cal{I}_Z}$ where $\cal{I}_Z$ is the ideal sheaf for $Z$ in $X$.
\begin{lem}
  We have an isomorphism $$H_{c,\nu}(W,X)^\mathrm{op} \simeq H_{\overline{c}, \omega_\mathrm{can} - \nu + \sum_{(w,Z)\in S(X)} 2c(w,Z)\omega_Z } (W,X),$$ given locally on an affine open $U\subset X$ by $x\mapsto x$ for $x\in \Gamma(U,\cal{O}_X)$, $D_y\mapsto -D_y$ for $y\in \Gamma(U,TX)$, and $g\mapsto g^{-1}$ for $g\in W$.
\end{lem}
\begin{proof} Beginning with the isomorphism $$j_* j^*(\cal{D}_X^\nu \rtimes \bC W)^\mathrm{op}\to j_* j^* (\cal{D}_X^{\omega_\mathrm{can} - \nu}\rtimes \bC W)$$ we must compute the image of a Dunkl operator $D_y$ for $y\in \Gamma(U,TX)$.  For every $(w,Z)\in S(X)$, let $f_Z \in \Gamma(U,\cal{O}_X(Z))$ be a function as in \cite[Definition 2.7]{E04}, and let $\lambda_{(w,Z)}$ be the nontrivial eigenvalue of $w$ on the conormal bundle to $Z$.  Letting $W_Z$ be the pointwise stabilizer of $Z$, set $$g_Z = \frac{1}{|W_Z|}\sum_{w \in W_Z} \lambda_{(w,Z)}w.f_Z.$$  Since $f_Z - g_Z \in \cal{O}_X$, we may write a Dunkl operator for $H_{c,\nu}(W,X)^\mathrm{op}$ on $U$ as $$D_y = \mathbb{L}_y + \sum_{(w,Z)\in S(X)} \frac{2c(w,Z)}{1-\lambda_{(w,Z)}} g_Z (w-1).$$

  Note that any $w\in W$ with $(w,Z)\in S(X)$ acts on $g_Z$ by $w.g_Z = \lambda_{(w,Z)}^{-1} g_Z$.  Thus we may calculate that the image of $D_y$,
  \begin{align*} & -\mathbb{L}_y + \sum_{(w,Z)\in S(X)} \frac{2c(w^{-1},Z)}{1-\lambda_{(w^{-1},Z)}} (w^{-1}-1) g_Z \\
    &= -\mathbb{L}_y + \sum_{(w,Z)\in S(X)} \frac{2\overline{c}(w,Z)}{1-\lambda_{(w,Z)}^{-1}} (w-1) g_Z \\
    &= -\mathbb{L}_y + \sum_{(w,Z)\in S(X)} \frac{2\overline{c}(w,Z)}{1-\lambda_{(w,Z)}^{-1}} g_Z (\lambda_{(w,Z)}^{-1} w-1) \\
    &= -\mathbb{L}_y - \sum_{(w,Z)\in S(X)} g_Z \left( \frac{2\overline{c}(w,Z)}{1-\lambda_{(w,Z)}} (w-1) +2c(w,Z) \right)
  , \end{align*}
  is a negative Dunkl operator $- D_y$ for the ``modified Cherednik algebra'' $H_{1,\overline{c}, \eta, \omega_\mathrm{can} - \nu}$ where $\eta(Z) = \sum_{w\in W_Z} 2c(w,Z)$.  Now we just apply \cite[Proposition 2.18]{E04}.
\end{proof}

Finally, suppose $\cal{L}$ is a $W$-equivariant line bundle on $X$.  Then we have $H_{c,\nu + \omega_\cal{L}}(W,X) \simeq \cal{L} \otimes_{\cal{O}_X} H_{c,\nu}(W,X) \otimes_{\cal{O}_X} \cal{L}^{-1}$. This gives a natural Morita equivalence of $H_{c,\nu}(W,X)$ with $H_{c,\nu + \omega_\cal{L}}(W,X)$: if $M$ is a module for $H_{c,\nu}(W,X)$, then $\cal{L} \otimes_{\cal{O}_X} M$ is a $H_{c,\nu + \omega_\cal{L}}(W,X)$-module.

As a result of the discussion above, we have a canonical Morita equivalence between $H_{c,\nu}(W,X)^\mathrm{op}$ and  $$H_{\overline{c}, - \nu + \sum_{(w,Z)\in S(X)} 2c(w,Z)\omega_Z } (W,X)$$ given by $M \mapsto K_X \otimes_{\cal{O}_X} M$.
\subsection{Pullback}
\label{sec:pull}

Suppose $X,Y$ are smooth $W$-varieties, and $\phi:X\to Y$ is a $W$-equivariant map of varieties.  Choose $c:S(Y)\to \bC$ to be a $W$-invariant function and choose $\omega \in \mathbb{H}^2(Y, \Omega_Y^{\geq 1})^W$.  Let $S_c(Y)$ be the set of all reflections $(w,Z)$ of $Y$ with $c(w,Z)\neq 0$.  We have a pullback functor in cases \ref{sec:melys} and \ref{sec:emb} below.
\subsubsection{Melys morphisms} \label{sec:melys} Suppose $\phi$ is flat and for all $(w,Z)\in S_c(Y)$, we have $\phi^{-1}(Z) \subset X^w$ (set-theoretically).  Such a map $\phi$ is called \emph{melys} (from Welsh) by Bellamy and Martino \cite[Definition 3.3.1]{BM}.  Pullback and pushforward for Cherednik modules are defined with respect to melys morphisms in sections 3.5 and 3.8 of that paper.

To illustrate that the second condition is not automatic, we give an example.
\begin{eg}
Let $X=\bC^2$, $Y=\bC$, and $X\to Y$ be the projection onto the first component.  Let $W=\bZ / 2\bZ$ act on both spaces by $-\mathrm{Id}$.  Now $\{0\} \subset Y$ is a reflection hyperplane so the condition is violated for any nonzero parameter $c:S(Y)\to \bC$.
\end{eg}

We may define a suitable pullback of $c$ to a function $\phi^* c$ on $S(X)$ as follows: $$\phi^* c(w,Z') = m c(w,Z),$$
where $Z$ is the closure of $\phi(Z')$ and $m$ is the scheme-theoretic multiplicity of $Z'$ in $\phi^{-1}(Z)$.  This is a simplification of the formula from section 3.3 of \cite{BM}, as the sum there contains just one term.

The proof of \cite[Proposition 3.4.1]{BM} shows we have a natural map of sheaves of left $H_{\phi^* c,\phi^* \omega}(W,X)$-modules $$H_{\phi^* c,\phi^* \omega}(W,X)\to \phi^* H_{c,\omega}(W,Y) =\cal{O}_X \otimes_{\phi^{-1} \cal{O}_Y} \phi^{-1} H_{c,\omega}(W,Y),$$ which allows us to define a pullback of sheaves of modules $$\phi^{0}: H_{c,\omega}(W,Y)-\mathrm{mod} \to H_{\phi^* c,\phi^* \omega}(W,X)-\mathrm{mod}$$ given by $$M\mapsto \phi^*(M) = \phi^* H_{c,\omega}(W,Y) \otimes_{\phi^{-1} H_{c,\omega}(W,Y)} \phi^{-1} M.$$

Since $\phi$ is flat, this pullback is exact.

\subsubsection{Closed embeddings} \label{sec:emb} Suppose $\phi$ is a closed embedding of smooth $W$-varieties.

In this situation, for each component $X' \subset \phi(X)$ the following condition holds: for every $(w,Z)\in S_c(Y)$, $X'$ is either contained in $Z$ or transverse to $Z$.

Some reflections of $X$ are naturally in one-to-one correspondence with the set of reflections of $Y$ which intersect $X$ transversely: for such a reflection $(w,Z)$ of $Y$, $(w,\phi^{-1}(Z))$ is a reflection of $X$.  Let $S'(X)$ be this set of reflections of $X$ obtained by restriction, and let $S''(X)$ denote the set of all other reflections of $X$.  To see that $S''(X)$ may be nonempty, we give an example.

\begin{eg}
Let $Y=\bC^2$, and let $W=\bZ / 2\bZ$ act by $-\mathrm{Id}$.  Then $Y$ has no reflections, but if $X$ is any line through the origin, then $\{0\}$ is a reflection hyperplane of $X$.
\end{eg}

We must define a suitable pullback of $c$ to a function $\phi^* c$ on $S(X)$.  For reflections $(w,\phi^{-1}(Z))\in S'(X)$ we define $\phi^* c(w,\phi^{-1}(Z))= c(w,Z)$, and for $(w,Z)\in S''(X)$, put $\phi^* c(w,Z)=0$.

It is a standard calculation to see that $\cal{O}_X \otimes_{\phi^{-1} \cal{O}_Y} \phi^{-1} H_{c,\omega}(W,Y)$ admits a natural left action of $H_{\phi^* c,\phi^* \omega}(W,X)$.  Thus we have a injective map of sheaves of left $H_{\phi^* c,\phi^* \omega}(W,X)$-modules $$H_{\phi^* c,\phi^* \omega}(W,X)\to \phi^* H_{c,\omega}(W,Y) =\cal{O}_X \otimes_{\phi^{-1} \cal{O}_Y} \phi^{-1} H_{c,\omega}(W,Y).$$  Finally we define a pullback of sheaves of modules $$\phi^{0}: H_{c,\omega}(W,Y)-\mathrm{mod} \to H_{\phi^* c,\phi^* \omega}(W,X)-\mathrm{mod}$$ given by $$M\mapsto \phi^*(M)= \phi^* H_{c,\omega}(W,Y) \otimes_{\phi^{-1} H_{c,\omega}(W,Y)} \phi^{-1} M.$$

In general, pullback is only right-exact in this setting.

\subsection{Pushforward}
In the situations of the previous section, $$\phi^* H_{c,\omega}(W,Y)$$ becomes a $(H_{\phi^* c,\phi^* \omega}(W,X), \phi^{-1} H_{c,\omega}(W,Y))$-bimodule.  Thus we also have a pushforward for right modules given as $$M\mapsto \phi_*(M\otimes_{H_{\phi^* c,\phi^* \omega}(W,X)}\phi^* H_{c,\omega}(W,Y)).$$

To define pushforward for left modules over the Cherednik algebra, we must use the equivalences of section \ref{sec:lr}.  Under these identifications, we may view $$K_X \otimes_{\cal{O}_X} \phi^* \left( H_{\overline{c},-\omega + \sum_{(w,Z)\in S(Y)} 2c(w,Z)\omega_Z}(W,Y)\right)\otimes_{\phi^{-1}\cal{O}_Y} \phi^{-1} K_Y^{-1}$$ as a $(\phi^{-1} H_{c,\omega}(W,Y),H_{\phi^* c,\phi^* \omega}(W,X))$-bimodule.  Notice that because of how we have defined $\phi^* c$ we automatically have $\phi^* \overline{c} = \overline{\phi^* c}$ and $\phi^* (\sum_{(w,Z)\in S(Y)} 2c(w,Z)\omega_Z) = \sum_{(w,Z)\in S(X)} 2\phi^*c(w,Z)\omega_Z$.  Calling this bimodule $T_{Y \leftarrow X}$ we may write the pushforward of left modules $$\phi_{0}: H_{\phi^* c,\phi^* \omega}(W,X)-\mathrm{mod} \to H_{c,\omega}(W,Y)-\mathrm{mod}$$ as $$\phi_0 : M\mapsto \phi_*(T_{Y \leftarrow X} \otimes_{H_{\phi^* c,\phi^* \omega}(W,X)} M).$$  Henceforth, we will only work with left modules.

It is easy to see that $\phi_0$ is right exact if $\phi$ is an affine map or a closed embedding.  As a special case, if $\phi$ is an open embedding between affine varieties, the functor $\phi_0$ will just be given by restriction from $H_{\phi^* c,\phi^* \omega}(W,X)$ to $H_{c,\omega}(W,Y)$, hence is exact.

\subsection{Fourier transform}
Let us introduce the Fourier transform for rational Cherednik algebras.  We have an isomorphism $$H_c(W,\fk{h}) \simeq H_c(W,\fk{h}^*)$$ given by $x\mapsto x$ for $x\in \fk{h}^*$, $y\mapsto -y$ for $y\in \fk{h}$, and $g\mapsto g$ for $g\in W$.  This isomorphism yields a functor $$\cal{F}_{\fk{h}}: H_c(W,\fk{h})-\mathrm{mod} \to H_c(W,\fk{h}^*)-\mathrm{mod}.$$

\begin{lem}
$\cal{F}_{\fk{h}}$ sends holonomic modules to holonomic modules.
\end{lem}
\begin{proof}
  This is clear since $\cal{F}$ preserves the arithmetic singular support under the isomorphism of Poisson varieties $(\fk{h}\oplus \fk{h}^*)/W \simeq (\fk{h}^*\oplus \fk{h})/W$ given by switching the two factors.
\end{proof}

\subsection{Kashiwara theorem}

Let $Y$ be a smooth variety with the action of a finite group $W$ and parameters $c, \omega$, and let $X\subset Y$ be a smooth, closed $W$-invariant subvariety.  Write $\phi:X\to Y$ for the inclusion.  For $y\in Y$, let $W_y$ be the stabilizer of $y$ in $W$; then $W_y$ acts linearly on the tangent space $T_y Y$.  Let $c_y$ be the function on the set of conjugacy classes of complex reflections in $W_y$ defined by $c_y(w) = c(w,Z_w)$, where $Z_w$ is the component of $Y^w$ which passes through $y$.



\begin{prop} \label{prop:kash}
Let $\phi:X\to Y$ be the inclusion of a $W$-stable smooth closed subvariety such that if $x$ is a generic point of any component of $X$, then $c_{\phi(x)}$ is a regular parameter for the action of $W_x$ on $T_{\phi(x)}Y$. The functor $\phi_0$ induces an equivalence of abelian categories between the category of sheaves of $H_{\phi^* c,\phi^* \omega}(W,X)$-modules and the category $H_{c,\omega}(W,Y)-\mathrm{mod}_X$ of sheaves of $H_{c,\omega}(W,Y)$-modules which are set-theoretically supported on $X$.
\end{prop}


\begin{proof}
  We define a functor $$\phi^{!}: H_{c,\omega}(W,Y)-\mathrm{mod} \to H_{\phi^* c,\phi^* \omega}(W,X)-\mathrm{mod}$$ given by $$M\mapsto \Hom_{\phi^{-1}\cal{O}_Y}(\cal{O}_X,\phi^{-1} M ) = \Hom_{\phi^{-1} H_{c,\omega}(W,Y)}(T_{Y \leftarrow X},\phi^{-1} M).$$  Just as in \cite[Proposition 1.5.25]{HTT}, $\phi^{!}$ is right adjoint to $\phi_0$.  Thus we have only to show that the adjunction morphisms
  $$\mathrm{Id} \to \phi^! \phi_0, \qquad \phi_0 \phi^! \to \mathrm{Id}$$
  are natural isomorphisms.  To do so, it will suffice to check that these are isomorphisms when restricted to formal neighborhoods around every point of $X/W$.  The functors $\phi_0, \phi^!$ are straightforward to define for formal schemes, and doing so they commute with restriction.  Thus using the fact that $W$ acts linearly on formal neighborhoods, along with the Bezrukavnikov-Etingof isomorphism \cite[Theorem 3.2]{BE}, we may replace $\phi$ with the inclusion $T_x X\to T_{\phi(x)}Y$, and replace $W$ with the stabilizer $W_x$.

  We have reduced to the case when $Y=\fk{h}$ is a vector space with a linear action of $W$ and $\phi$ is the inclusion of a $W$-stable subspace $X$.  Since $H_{c_{\phi(x)}}(W,\fk{h}) = H_{c_{\phi(x)}}(W,\fk{h}/X) \otimes H_{c_{\phi(x)}}(W,X)$, we may reduce further to the case that $X$ is the origin of $\fk{h}$.

  Let $\cal{O}^-$ be the Fourier dual of the ordinary category $\cal{O}_{c_{\phi(x)}}$ studied in \cite{GGOR}.  If $M\in H_{c_{\phi(x)}}(W,\fk{h})-\mathrm{mod}_{\{0\}}$, then by the PBW theorem, $M$ is finitely generated over $H_{c_{\phi(x)}}$ if and only if $M$ is in category $\cal{O}^-$, that is, $M$ is finitely generated over the subalgebra $\bC[y_1, \ldots, y_n]$ of $H_{c_{\phi(x)}}$ generated by the Dunkl operators and the $x_i$ act locally nilpotently on $M$.  Now we have assumed that $c_{\phi(x)}$ is a regular for a generic point of each component of $X$, which implies that $c_{\phi(x)}$ is regular for all closed points $x$ of $X$.  The theorem follows, since category $\cal{O}^-$ is semisimple and any $H_{c_{\phi(x)}}(W,\fk{h})$-module supported at the origin is a direct limit of modules in category $\cal{O}^-$.
\end{proof}

As we noted, the theorem holds in particular if $X$ is transverse to each $Z$ with $(w,Z)\in S_c(Y)$ for some $w\in W$, since each $c_{\phi(x)}=0$ in this case.  Indeed, the above transversality condition is satisfied if and only if no component of $X$ is contained in any $Z$ with $(w,Z)\in S_c(Y)$ for some $w\in W$.

It is easy to see that the functors $\phi_0$ and $\phi^!$ preserve the full subcategories of $H_c$-coherent modules, so the equivalence restricts to these categories as well.




\begin{prop} \label{prop:pfhol} Let $\phi:X \to Y$ be as in Proposition \ref{prop:kash}.  The functor $$\phi_{0}:H_{\phi^* c,\phi^* \omega}(W,X)-\mathrm{mod} \to H_{c,\omega}(W,Y)-\mathrm{mod}_X$$ and its inverse $\phi^{!}$ preserve holonomicity.
\end{prop}
\begin{proof}
Write $V'(M)$ to refer to the preimage of $V(M)$ under $T^* X \to T^* X / W$.  Let $\varpi: X \times_Y T^* Y \to T^* Y$ and $\rho: X \times_Y T^* Y \to T^* X$ be the canonoical maps.  Then we have $V'(\phi_0 M) = \varpi \rho^{-1} (V'(M))$.  Hence $V'(\phi_0 M)$ is isotropic if and only if $V'(M)$ is.  This suffices, for $V'(M)$ is isotropic if and only if $M$ is holonomic by \cite[Proposition 1.3.30]{CG}.
\end{proof}

\begin{rmk}
  Finally, suppose that $\phi:X\to Y$ is \emph{any} inclusion of a $W$-stable smooth closed subvariety.  If we have that $N \in H_{c,\omega}(W,Y)-\mathrm{mod}_X$ is a quotient of $\phi_0 M$ for some holonomic $H_{\phi^* c,\phi^* \omega}(W,X)$-module $M$, then \cite[Proposition 1.3.30]{CG} and the proof of Proposition \ref{prop:pfhol} immediately show that $\phi^{!} N$ is holonomic.  We shall later use this to show that $\phi^{!}$ preserves holonomicity in this generality.
\end{rmk}

\subsection{Verdier dual}
Notation: if $A^\bullet$ is a complex, we define the shift functor $A[i]^\bullet$ by $A[i]^j = A^{i+j}$.

In this section we return to the setting when $H_c = H_c(W,\fk{h})$ is a rational Cherednik algebra and $\dim \fk{h} = r$ (one can define the duality functor in the global situation as well).  Let $D^b_f(H_c)$ denote the bounded derived category of finitely generated $H_c$-modules.

We have a duality functor on the derived category of Cherednik modules given by $\bD(\cdot) = \Hom_{D^b_f(H_c)}(\cdot,H_c)[r]$.  We have $\bD(\bD(\cdot)) \simeq \mathrm{Id}$, see \cite[Proposition D.4.1]{HTT}, so $\bD$ is an antiequivalence of $D^b_f(H_c)$ with $D^b_f(H_{\overline{c}})$ (here we tacitly use the equivalence between right $H_c$ and left $H_{\overline{c}}$ modules discussed in section \ref{sec:lr}).  We have already seen that in the case of generic parameters that $\bD$ takes holonomic modules concentrated in degree zero to holonomic modules in degree 0.  Moreover, for arbitrary parameters, \cite[Proposition 4.10]{GGOR} shows that $\bD$ preserves the bounded derived category of $\cal{O}$, namely
$$\bD: D^b(\cal{O}_c) \to D^b(\cal{O}_{\overline{c}})^\mathrm{op}$$
is an equivalence.

For arbitrary parameters, if $M$ is holonomic, concentrated in degree 0, then Proposition \ref{prop:exts} says that $H^i (\bD(M)) = 0$ if $i<j(M) - r$ or $i>r$ (the second inequality is because $H_c$ has homological dimension at most $2r$).  However, if $M$ is a simple holonomic module, it is not necessarily the case that $\bD(M)$ has cohomology concentrated in a single degree.  Here is an example, pointed out to me by Etingof and Bellamy.

\begin{eg}
Take $W=S_3$, $\fk{h}$ the reflection representation so $\dim \fk{h}=2$, and $c=1/3$.  Let $M$ be the augmentation ideal of the polynomial representation (i.e., the radical of the Verma module for the trivial representation of $W$) of $H_c$.  Then $\bD(M)$ has nonzero cohomology in multiple degrees.  To see this, observe that $M$ is not Cohen-Macaulay over $\bC[\fk{h}]$, and \cite[Corollary 3.3]{EGL} implies that a simple module in category $\cal{O}$ has dual concentrated in single degree if and only if it is Cohen-Macaulay over $\bC[\fk{h}]$.
\end{eg}

However, we do have the following result.

\begin{lem}
  If $M$ is a holonomic $H_c$-module, then $H^i(\bD(M))$ is holonomic for each $i$.
\end{lem}
\begin{proof}
  We have $V(M) = V(\bD(M))$, and the latter, by definition, is just $\bigcup_i V(H^i(\bD(M)))$ (see \cite[Proposition D.4.2]{HTT}).  Since $V(H^i(\bD(M))) \subset V(M)$ and $V(M)$ is isotropic in every symplectic leaf, so also is $V(H^i(\bD(M)))$.  This gives the result.
\end{proof}
\subsection{An adjunction}
Given an affine morphism $\phi$ which is either melys or a closed embedding, we write $\phi_\bullet = L \phi_0, \phi^\bullet = L \phi^0$.  These are functors on the respective derived categories of Cherednik modules.  If $j:U\to X$ is an open embedding of affine varieties, however, $j_0$ and $j^0$ are exact.  Write $D^b(H_{c,\omega})$ to denote the bounded derived category of $H_{c,\omega}(W,X)$-modules.

\begin{lem}
  Let $j:U\to X$ be a $W$-equivariant affine open embedding, $X$ affine.  In the derived category of Cherednik modules, $j_\bullet$ is right adjoint to $j^\bullet$.
\end{lem}
\begin{proof}
If $N$ is an $H_{j^* c,j^* \omega}(W,U)$-module then $j_0 N$ is simply the restriction of $N$ to $H_{c,\omega}(W,X) \subset H_{j^* c,j^* \omega}(W,U)$.  Moverover, if $M$ is an $H_{c,\omega}(W,X)$-module, then $j^0 M = H_{j^* c,j^* \omega}(W,U) \otimes_{H_{c,\omega}(W,X)} M$.  Note that $H_{j^* c,j^* \omega}(W,U)$ is flat over $H_{c,\omega}(W,X)$, so we have $$\Ext^i_{H_{c,\omega}(W,X)}(M,j_0 N) = \Ext^i_{H_{j^* c,j^* \omega}(W,U)}(j^0 M, N).$$  Since the functors $j_0$ and $j^0$ are exact in this case, this proves the lemma.
\end{proof}

\section{Preservation of holonomicity}
\subsection{Generic parameters}

In this section, $H_c = H_c(W,\fk{h})$ is a rational Cherednik algebra.  The following results we prove in the case when the parameter $c$ is regular.  Let $U\subset \fk{h}$ be a $W$-stable, affine open subvariety.  Let $j:U\to \fk{h}$ be the inclusion, and $r=\dim \fk{h}$.

\begin{thm} \label{thm:hol} If $M$ is a holonomic $H_c(W,U)$-module then $j_{0}(M)$ is holonomic.
\end{thm}
\begin{proof} There is a $W$-invariant polynomial $\delta$ on $\fk{h}$ such that $U$ is the nonzero locus of $\delta$.  We factor $j=\pi \circ i$ where $i:U\to \fk{h}\oplus \bC$ is given by $i(v)=(v, \delta(v)^{-1})$, and $\pi:\fk{h}\oplus \bC \to \fk{h}$ is the projection.  Note, the action of $W$ on the second factor of $\fk{h}\oplus \bC$ is trivial.  Now $j_0 = \pi_0 \circ i_0$.  We know by Proposition \ref{prop:pfhol} that $i_0$ preserves holonomicity.

  We claim that $\cal{F}_{\fk{h}} \circ \pi_0 = (\pi^\vee)^0 \circ \cal{F}_{\fk{h}\oplus \bC}$ where $\pi^\vee: \fk{h}^* \to \fk{h}^*\oplus \bC$ is the dual map.  Indeed for any module over $H_c(W,\fk{h}\oplus \bC)$, it is easy to see that the images under the two functors agree (under the identification of left and right modules).  Thus we have only to show that $(\pi^\vee)^0$ preserves holonomicity, so suppose $N$ is a holonomic $H_c(W,\fk{h}^* \oplus \bC)$-module.  Let $F$ be a good filtration on $N$ with respect to the Bernstein filtration, so we have $\dim F_j(N) = c_{r+1} j^{r+1} + c_r j^r + \cdots + c_0$ for $j\gg 0$ and for some constants $c_i$.  Since $(\pi^\vee)^0$ is right exact, we may assume $N$ is irreducible.  Let $x:\fk{h}^*\oplus \bC \to \bC$ be the second projection.  We have $(\pi^\vee)^0 N = N/xN$.  Now $F$ induces a filtration on $N/xN$.  If we assume that multiplication by $x$ on $N$ is injective, we have $\dim F_j (N/xN) \leq \dim F_j(N) - \dim F_{j-1}(N) \leq c j^r$ for $j\gg 0$ where $c>(r+1) c_{r+1}$.  This module has GK-dim equal to $r$, hence $N/xN$ is holonomic by Proposition \ref{prop:gkdim} in this case.

  Finally, we must consider the case when multiplication by $x$ on $N$ is not injective, and so $N$ is set-theoretically supported on $\fk{h}^* \subset \fk{h}^*\oplus \bC$ since we have assumed $N$ to be irreducible.  Note that we have the decomposition $H_c(W,\fk{h}^* \oplus \bC) \simeq H_c(W,\fk{h}^*) \otimes \cal{D}(\bC)$.  The restriction of $N$ to $\cal{D}(\bC)$ is supported at $0\in \bC$.  Now the usual Kashiwara theorem $\cal{D}$-modules, applied to the inclusion $\{0\} \to \bC$, gives us that $N/xN = 0$.  Hence $(\pi^\vee)^0 N = 0$ is holonomic.
\end{proof}

\subsection{Rank 1 case, arbitrary parameters}
Let $\fk{h}=\bC$, $W=\bZ/ m\bZ$ the group of $m^\mathrm{th}$ roots of unity.  Choose $\lambda$ a primitive $m^\mathrm{th}$ root of unity, and write $W=\{s_i| i=1, \ldots, m\}$ so that $s_i$ acts on $\fk{h}$ by $\lambda^{-i}$.  In this case we have $\fk{h}^\mathrm{reg}=\bC^*$.  If $\cal{C}$ is the Serre subcategory of $H_c-\mathrm{mod}$ consisting of modules supported at the origin, then in this case $H_c-\mathrm{mod}/\cal{C}$ is naturally equivalent to the category of $H_c(W,\bC^*)$-modules, which, since $H_c(W,\bC^*)=\cal{D}(\bC^*) \rtimes \bC W$, this is just the category of $W$-equivariant $\cal{D}$-modules on $\bC^*$.  The quotient functor $H_c-\mathrm{mod} \to H_c-\mathrm{mod}/\cal{C}$ is given by localization, and the functor $j_0$ of extension by poles is a section, up to isomorphism.

If $M$ is a holonomoic $H_c(W,\bC^*)$-module, we can easily compute the singular support of $j_{0}(M)$, showing it is holonomic.  This approach has the added benefit that one is able to describe the subquotients of $j_{0}(M)$ and, in principle, compute its length.

We study the case when $M$ is an irreducible rank one $W$-equivariant De Rham local system on $\fk{h}^\mathrm{reg}$, so $M=\bC[x^{\pm 1}]$ as an $\bC[\fk{h}^\mathrm{reg}]$-module.  Then $j_0 M$ is a module over $H_c$ where $y$ acts by
$$ \partial_x + p(x) + \sum_{i=1}^{m-1} \frac{2c_i}{(1-\lambda^i)x} (1-s_i),$$
where $p(x) \in \bC[x^{\pm 1}]$ is such that $xp \in \bC[\fk{h}^\mathrm{reg}]^W$.

\begin{prop} \,
  \begin{enumerate}
  \item $j_0 M$ is reducible if and only if $xp$ has constant term divisible by $m$ and no terms of negative degree.
  \item If so, then $j_0 M$ has a unique irreducible submodule which has full support, and the quotient by which is in $\cal{O}^{-}$ (the Fourier dual of the ordinary category $\cal{O}$ studied in \cite{GGOR}).
  \end{enumerate}
\end{prop}
\begin{proof}  
  If $xp$ has constant term divisible by $m$ and no terms of negative degree, then we may change basis of $M$ to assume that $p\in x^{m-1} \bC[\fk{h}]^W$ in the following way.  If the constant term is $mk$ then we make the change of basis $x^i \mapsto x^{i+mk}$.  In this new basis, $y$ acts by
  $$ \partial_x - mk x^{-1} + p(x) + \sum_{i=1}^{m-1} \frac{2c_i}{(1-\lambda^i)x} (1-s_i).$$
  In this case it is clear that $\bC[x]\subset j_0 M$ is a proper submodule.

  Then for the second statement of the Proposition, let us call this submodule $N$.  If $N$ is irreducible then the second statement is clear.  If $N$ is not irreducible, then it has a finite-dimensional quotient.  Now again, the statement is clear since $\cal{O}^-$ contains all finite dimensional modules. 
  
  The converse of the first statement will follow from:
  \begin{lem}
    If $N \subset j_0 M$ is a proper submodule, then $N$ is coherent over $\bC[\fk{h}]$.
  \end{lem}
  \begin{proof}
    It is not hard to see that $M=\bC[x^{\pm 1}]$ is finitely generated over $H_c$, generated, for example, by $x^{-k}$ for large enough $k$.  So we may choose a good filtration on $M$ with respect to the geometric filtration on $H_c$: let $F_i = x^{-k-di} \bC[x]$, where $d$ is the lowest degree of $p$.  From this, we see that the preimage of the geometric singular support in $\fk{h} \times \fk{h}^*$ is the union of the zero section and the fiber over $0\in \fk{h}$, each having multiplicity 1.  Now since the singular cycle is additive over short exact sequences, and because $j_0 M$ has no submodule supported at $0\in \fk{h}$ or finite dimensional quotient, we see that the singular support of $N$ is just the zero section.  Hence $N$ is coherent over $\bC[\fk{h}]$.
  \end{proof}
  By the Lemma, if $j_0 M$ is reducible, then it has a submodule equal to $x^k \bC[x]$.  But this means $xp$ has constant term divisible by $m$ and no terms of negative degree.
\end{proof}

\printbibliography
\end{document}